\documentclass[12pt]{article}

\usepackage{amsthm,tikz,amssymb, latexsym, amsmath, amscd, amsfonts, array, graphicx, lmodern, slantsc,enumerate,stmaryrd,rotating}

\usepackage{hyperref}
\hypersetup{nesting=true,debug=true,naturalnames=true}
\usepackage[all]{xy}
\CompileMatrices

\usepackage{color}

\def\imm{\protect\operatorname{Imm}}

\def\RP{\mathbb{R}\mbox{\textrm P}}

\def\TC{\protect\operatorname{TC}}

\def\zcl{\protect\operatorname{zcl}}

\newtheorem{proposition}{Proposition}[section]
\newtheorem{corollary}[proposition]{Corollary}

\newtheorem{theorem}[proposition]{Theorem}

\newtheorem{lemma}[proposition]{Lemma}

\begin{document}

\title{Sequential motion planning in connected sums of real projective spaces}
\author{Jorge Aguilar-Guzm\'an and Jes\'us Gonz\'alez}
\date{\empty}

\maketitle

\begin{abstract}
In this short note we observe that the higher topological complexity of an iterated connected sum of real projective spaces is maximal possible. Unlike the case of regular TC, the result is accessible through easy mod 2 zero-divisor cup-length considerations.
\end{abstract}



\section{Introduction}
It was proved in~\cite{MR1988783} that the topological complexity (TC) of the $m$-th dimensional real projective space $\RP^m$ agrees\footnote{This characterization holds as long as $\RP^m$ is not parallelizable: $\TC(\RP^m)=\imm(\RP^m)-1=m$ for $m=1,3,7$.} with $\imm(\RP^m)$, the minimal dimension $d$ so that $\RP^m$ admits a smooth immersion in $\mathbb{R}^d$. Cohen and Vandembroucq have recently shown in~\cite{connectedsums} that the fact above does not hold for $g\RP^m$, the $g$-iterated connected sum of $\RP^m$ with itself, if $g\geq2$. Indeed, $\TC(g\RP^m)$ is maximal possible whenever $g\ge2$, a result that contrasts with the currently open problem of assessing how much $\TC(\RP^m)$ deviates from~$2m$.

Cohen and Vandembroucq's result for $\TC(g\RP^m)$ extends their impressive calculation in~\cite{LDKlein}, using obstruction theory, of the topological complexity of non orientable closed surfaces. In this short note we observe that a simple minded zero-divisor cup-length argument suffices to prove the analogous fact for Rudyak's higher topological complexity $\TC_s$:

\begin{theorem}\label{maintheorem}
For $g,m\ge2$ and $s\ge3$, $\TC_s(g\mathbb{R}\mbox{\emph{P}}^m)=sm$.
\end{theorem}

This is the same (but much simplified) phenomenon for $\TC_s(\RP^m)$ studied in~\cite{MR3770001,MR3795624}. The case $m=2$ is essentially contained in~\cite[Proposition~5.1]{MR3576004}.

\section{Proof}
We assume familiarity with the basic ideas, definitions and results on Rudyak's higher topological complexity, a variant of Farber's original concept (see~\cite{MR3331610}). In what follows all cohomology groups are taken with mod 2 coefficients. 

The first ingredient we need is the well-known description of the cohomology ring of the connected sum $M\#N$ of two $n$-manifolds $M$ and $N$: Using the cofiber sequence 
$$
S^{n-1}\hookrightarrow M\#N\to M\vee N
$$
one can see that the cohomology ring $H^*(M\#N)$ is the quotient of $H^*(M\vee N)$ by the ideal generated by the sum $[M]^*+[N]^*$ of the duals of the (mod 2) fundamental classes of $M$ and $N$. In particular, for the $g$-iterated connected sum $g\RP^m$ of $\RP^m$ with itself, we have:

\begin{lemma}\label{ing2}
The cohomology ring of $g\mathbb{R}\emph{P}^m$ is generated by 1-dimensional cohomology classes $x_u$, for $1\le u\le g$, subject to the three relations:
\begin{itemize}
\item $x_ux_v=0$, for $u\neq v$;
\item $x_u^{m+1}=0$;
\item $x_u^m=x_v^m$.
\end{itemize}
\end{lemma}

The top class in $H^*(g\mathbb{R}\mathrm{P}^m)$ is denoted by $t$; it is given by any power $x_u^m$ with $1\le u\le g$.

\begin{corollary}
The cohomology ring of the $s$-fold cartesian product of $g\mathbb{R}\mathrm{P}^m$ with itself is given by
\begin{equation}\label{topclass}
H^*(g\mathbb{R}\emph{P}^m\times\cdots\times g\mathbb{R}\emph{P}^m)\cong\bigotimes_{j=1}^s\left(\mathbb{Z}_2[x_{1,j},\ldots,x_{g,j}] / I_{{g,j}}\rule{0mm}{4mm}\right).
\end{equation}
Here $x_{u,j}$ is the pull back of $x_u\in H^1(\mathbb{R}\emph{P}^m)$ under the $j$-projection map $\left(\mathbb{R}\emph{P}^m\right)^{\times s}\to\mathbb{R}\emph{P}^m$, and $I_{g,j}$ is the ideal generated by the elements $x_{u,j}^{m+1}$, $x_{u,j}^m+x_{v,j}^m$ and $x_{u,j}x_{v,j}$ for $u\neq v$.
\end{corollary}

We let $t_j\in H^m(\left(\mathbb{R}\mathrm{P}^m\right)^{\times s})$ stand for the image of the top class $t\in H^m(\mathbb{R}\mathrm{P}^m)$ under the $j$-th projection map $\left(\mathbb{R}\mathrm{P}^m\right)^{\times s}\to\mathbb{R}\mathrm{P}^m$. The top class in~(\ref{topclass}) is then the product $t_1 t_2\cdots t_s$, which agrees with any product $x_{u_1,1}^m x_{u_2,2}^m\cdots x_{u_s,s}^m$.

The second ingredient we need concerns with standard estimates for the higher topological complexity of CW complexes:
\begin{lemma}[{\cite[Theorem~3.9]{MR3331610}}]\label{ing1}
For a path connected CW complex $X$, $$\zcl_s(X)\le\TC_s(X)\leq s\dim(X),$$ where $\zcl_s(X)$ is the maximal length of non-zero cup products of $s$-th zero divisors, i.e., of elements in the kernel of the $s$-iterated cup-product map $H^*(X)^{\otimes s}\to H^*(X)$.
\end{lemma}

Note that any element $x_{r,i}+x_{r,j}$ is a zero-divisor, so that Theorem~\ref{maintheorem} follows from:
\begin{proposition}
The product
$$
(x_{1,1}+x_{1,2})^m (x_{1,1}+x_{1,3})^m \cdots (x_{1,1}+x_{1,s})^m (x_{2,1}+x_{2,2})^{m-1} (x_{2,1}+x_{2,3})
$$
is the top class in $H^*((g\mathbb{R}\emph{P}^m)^{\otimes s})$ provided $g,m\ge2$ and $s\ge3$.
\end{proposition}
\begin{proof}
The case $s=3$ follows from a direct calculation: 
\begin{align}
(x_{1,1}+x_{1,2})^m & (x_{1,1}+x_{1,3})^m (x_{2,1}+x_{2,2})^{m-1} (x_{2,1}+x_{2,3})\nonumber\\
& {}=(x_{1,1}^m+\cdots+x_{1,2}^m) (x_{1,1}^m+\cdots+x_{1,3}^m) (x_{2,1}^{m-1}+\cdots+x_{2,2}^{m-1}) (x_{2,1}+x_{2,3}),\label{uno}\\
& {}=(x_{1,1}^m+\cdots+x_{1,2}^m) x_{1,3}^m (x_{2,1}^{m-1}+\cdots+x_{2,2}^{m-1}) (x_{2,1}+x_{2,3})\label{dos}\\
& {}=(x_{1,1}^m+\cdots+x_{1,2}^m) x_{1,3}^m (x_{2,1}^{m-1}+\cdots+x_{2,2}^{m-1}) x_{2,1}\label{tres}\\
& {}=x_{1,2}^m x_{1,3}^m (x_{2,1}^{m-1}+\cdots+x_{2,2}^{m-1}) x_{2,1}\label{cuatro}\\
& {}=x_{1,2}^m x_{1,3}^m x_{2,1}^{m-1} x_{2,1}\label{cinco}\\
& {}=x_{1,2}^m x_{1,3}^m x_{2,1}^m\,=\,t_1t_2t_3.\nonumber
\end{align}
Note that equality in~(\ref{dos}) holds because of the description of $I_{g,s}$: the factor $t_3$ in the top class $t_1t_2t_3$ can only arise from the summand $x_{1,3}^m$ in the second factor of~(\ref{uno}). Likewise, equality in~(\ref{tres}) comes from the relation $x_{1,3}x_{2,3}=0$, equality in~(\ref{cuatro}) comes from the relation $x_{1,1} x_{2,1}=0$, and equality in~(\ref{cinco}) comes from the relation $x_{1,2} x_{2,2}=0$.

The general case then follows easily from induction:
\begin{align*}
(x_{1,1}+x_{1,2})^m & (x_{1,1}+x_{1,3})^m \cdots (x_{1,1}+x_{1,s+1})^m (x_{2,1}+x_{2,2})^{m-1} (x_{2,1}+x_{2,3})\\
& {}=t_1\cdots t_s (x_{1,1}+x_{1,s+1})^m = t_1\cdots t_s x_{1,s+1}^m = t_1\cdots t_{s+1},
\end{align*}
where the next-to-last equality holds in view of the relation $x_{1,1}^{m+1}=0$.
\end{proof}


\bigskip\smallskip

{\small \sc Departamento de Matem\'aticas

Centro de Investigaci\'on y de Estudios Avanzados del I.P.N.

Av.~Instituto Polit\'ecnico Nacional n\'umero 2508

San Pedro Zacatenco, M\'exico City 07000, M\'exico

{\tt jesus@math.cinvestav.mx}

{\tt  jaguzman@math.cinvestav.mx}}
\end{document}